\documentclass{amsart}

\usepackage[latin1]{inputenc}   
\usepackage[T1]{fontenc}       
\usepackage{amsmath}
\usepackage{amssymb}
\usepackage{amsthm}

\usepackage[dvipsnames]{xcolor}

\usepackage{bm}
\usepackage{mathtools}
\usepackage{mathrsfs}
\usepackage{accents} 
\usepackage{hyperref}
\hypersetup{pdfstartview=}

\newcommand{\beq}{\begin{equation}}
\newcommand{\eeq}{\end{equation}}

\newcommand{\Z}{\ensuremath{\mathbb{Z}}}
\newcommand{\R}{\ensuremath{\mathbb{R}}}

\newcommand{\SL}{\ensuremath{\mathrm{SL}}}

\renewcommand{\epsilon}{\varepsilon}

\newcommand{\ab}{\ensuremath{\mathrm{ab}}}

\newcommand{\diam}{\operatorname{diam}}
\newcommand{\sgn}{\operatorname{sgn}}

\title{Diameters of random Cayley graphs of finite nilpotent groups}
\author{Daniel El-Baz, Carlo Pagano}
\address{Institute of Analysis and Number Theory, TU Graz \\ Steyrergasse 30  \\ 8010 Graz \\ Austria}
\email{danielelbaz88@gmail.com}

\address{Max Planck Institute for Mathematics \\ Vivatsgasse 7 \\ 53111 Bonn \\ Germany}
\email{carlein90@gmail.com}

\begin{document}

\begin{abstract}
We prove the existence of a limiting distribution for the appropriately rescaled diameters of random undirected Cayley graphs of finite nilpotent groups of bounded rank and nilpotency class, thus extending a result of Shapira and Zuck which dealt with the case of abelian groups. The limiting distribution is defined on a space of unimodular lattices, as in the case of random Cayley graphs of abelian groups. Our result, when specialised to a certain family of 
unitriangular groups, establishes a very recent conjecture of Hermon and Thomas. We derive this as a consequence of a general inequality, showing that the diameter of a Cayley graph of a nilpotent group is governed by the diameter of its abelianisation. 
\end{abstract}

\maketitle

\newtheorem{defn}{Definition}[section]
\newtheorem{prop}{Proposition}[section]
\providecommand*{\propautorefname}{Proposition}
\newtheorem{thm}{Theorem}[section]
\providecommand*{\thmautorefname}{Theorem}
\newtheorem{lemma}{Lemma}[section]
\providecommand*{\lemmaautorefname}{Lemma}
\newtheorem{cor}{Corollary}[section]
\providecommand*{\corautorefname}{Corollary}
\newtheorem{rmk}{Remark}[section]
\providecommand*{\rmkautorefname}{Remark}

\section{Introduction} \label{intro}
Metric properties of graphs are important in the study of networks. 
A key example is given by the diameter of a graph, which is defined to be the longest distance among the pairs of vertices of the graph.

A natural family of graphs is provided by Cayley graphs of groups.
For certain finite simple groups and generating sets, upper bounds on the diameter show logarithmic growth. That is sharp since one always has a logarithmic lower bound, which essentially comes from the fact that finitely generated groups always have at most exponential growth. The motivation for proving such upper bounds is Babai's conjecture \cite[Conjecture 1.7]{BS1992}, which postulates the existence of a constant $a>0$ such that for every finite simple group $G$ and every generating set $S$, one has $\diam(\Gamma(G,S)) \le (\log{|G|})^{a}$.

In contrast, Amir and Gurel-Gurevich \cite{AGG2010} started investigating the diameter of cyclic groups $\Z/q\Z$ with respect to a random set $S$ of generators of fixed size, say $k \ge 2$. The fact that finitely generated abelian group of rank $k$ have growth of order a polynomial of degree $k$ in the radius is reflected in a simple lower bound of the order of $q^{1/k}$; furthermore, they obtain that for any function $f$ going to infinity with $q$ the probability that the diameter of the random Cayley graph is bigger than $f(q) q^{1/k}$ goes to $0$ as $q \to \infty$. That led them to conjecture that, as $q \to \infty$, the random variables given by the diameter of the corresponding random Cayley graph, when rescaled by $q^{1/k}$, converge in distribution. 

Marklof and Str\"ombergsson \cite{MarklofStrombergsson2013} introduced a strategy relating that problem to an equidistribution theorem in homogeneous dynamics and were able to prove a version of that conjecture in which the cyclic group itself was also taken at random (with $q \in \Z \cap [1, Q]$).
Inspired by that approach, Shapira and Zuck  \cite{ShapiraZuck2018} settled that conjecture and further extended it from finite cyclic groups to arbitrary finite abelian groups of bounded rank.
In the present article, we obtain the analogous result for finite nilpotent groups with bounded rank and nilpotency class.

Recall that for a group $G$ one can inductively define the filtration of subgroups $\mathbb{Z}_{\ge 1} \ni i \mapsto G^{(i)}$--- called the lower central series --- given by
$G^{(1)}=G$
and for $i \ge 1$,
\begin{equation*}
G^{(i+1)}=[G,G^{(i)}].
\end{equation*}
A group is said to be nilpotent if there exists $i \ge 1$ such that $G^{(i)}=\{\mathrm{id}\}$. In that case, the nilpotency class of $G$ is defined to be the smallest positive integer $c$ such that $G^{(c+1)}=\{\mathrm{id}\}$.
For a group $G$ and a symmetric generating set $S \subset G$, we denote by $\Gamma(G, S)$ the Cayley graph of $G$ with respect to $S$.

\begin{thm} \label{thm:main_intro}
Let $k > r \geq 1$ and $c \geq 1$ be integers. Let $\{G_n\}_{n \in \mathbb{Z}_{\geq 1}}$ be a sequence of finite nilpotent groups of rank at most $r$, nilpotency class at most $c$ and with $\lim_{n \to \infty} |G_n| = \infty$.
Choosing a subset $S$ uniformly at random among all symmetric generating subsets of $S$ of $G_n$ of size $k$, then as $n \to \infty$ the random variables $\dfrac{\diam(\Gamma(G_n, S))}{|G_n^{\ab}|^{\frac{1}{k}}}$ converge in distribution.
\end{thm}
A version of this theorem which also contains a fairly explicit description of the limiting distribution is given as \autoref{thm:main1}.

In recent work, Hermon and Thomas \cite{HermonThomas2019} investigated random walks on certain finite unitriangular groups, defined for $q, d \in \Z_{\ge 2}$, to be the group of $d \times d$ matrices over $\Z/q\Z$ which are upper triangular and whose diagonal entries are all $1$; that group is denoted by $H_{q,d}$.
Those are special examples of finite nilpotent groups.

Hermon and Thomas establish a concentration for the {\em typical distance} --- a function of a parameter $\beta \in (0,1)$ defined as the smallest radius of a ball centred at the identity which is enough to cover a proportion $\beta$ of the group --- of the random Cayley graphs of those unitriangular groups, which they show concentrates around the value it takes for the abelianisation of $H_{q,d}$ when the number of generators diverges (or is at least large enough as a function of $d$); combined with the simple lower bound on the diameter coming from the growth of the group, which is of the same order, that led them to conjecture the existence of a limiting distribution for the diameters of those graphs with that particular rescaling.

As a consequence of \autoref{thm:main_intro}, we establish their conjecture. \footnote{Their paper only deals with $q$ prime and directed graphs; ours treats arbitrary $q \ge 2$ but undirected graphs.}

\begin{thm}
Let $q\ge 2, d \ge 2$ and $k \ge d$.
Let $Z_1(q), \ldots, Z_k(q)$ be generators of $H_{q,d}$ chosen uniformly and independently, write $\Delta_{Z(q)}(k)$ for the diameter of the random Cayley graph with generators $Z(q) = (Z_1(q)^{\pm 1}, \ldots, Z_k(q))^{\pm 1})$.
As $q \to \infty$, the random variables $\dfrac {\Delta_{Z(q)}(k)}{q^{(d-1)/k}}$ converge in distribution.
\end{thm}

We state a more precise version of the above as \autoref{thm:main} which also includes an explicit description of the limiting distribution in terms of the space of $k$-dimensional unimodular lattices. The latter is the same as the limiting distribution for the random undirected Cayley graph of the finite abelian group $(\Z/q\Z)^{d-1}$ with $k$ generators chosen uniformly at random which was obtained by Shapira and Zuck.
Note that the distribution is also the same as that from the paper by Marklof and Str\"ombergsson (for random undirected circulant graphs with respect to $k$ generators), in which they make use of the description in terms of random unimodular lattices in $\R^k$ to derive quantitative properties of the limiting distribution such as tail estimates.

Indeed, our strategy for proving those theorems consists in establishing a general inequality for the diameter of a Cayley graph on a finite nilpotent group, which essentially shows that this diameter is governed by the diameter of the abelianisation. This is done in \autoref{sec:finnil}. The crucial step is \autoref{prop:main}. In that proposition we take advantage of the well-known phenomenon of \emph{distortion} in nilpotent groups, that is the possibility of rewriting $N$ times a nested commutator of length $i$ in time $O(N^{\frac{1}{i}})$ modulo nested commutators of length at least $i+1$, for a positive integer $N$. The upper bound in \autoref{prop:main} is reminiscent of the formula of Bass and Guivarc'h for the growth in finitely generated nilpotent group \cite[Appendix]{Gromov1981}, which indeed relies on the same phenomenon of distortion. We remark that a very similar argument can be found in the proof of \cite[Lemma 4.11]{BreuillardTointon2016}. \footnote{We thank Matthew Tointon for pointing out this reference to us after receiving a first draft of this paper.} This upper bound leads us to wonder whether $q^{\frac{d-1}{ik}}$ is the correct scale for the diameters (with respect to the ambient metric on the group) of the $i$-th term of the lower central series of undirected Cayley graphs of $H_{q,d}$ with respect to a random generating set. In the concluding section we ask this and a few related questions.

{\bf Acknowledgements:} We thank Jonathan Hermon for a helpful email on his work with Sam Thomas \cite{HermonThomas2019}. We also thank him for pointing out that in a previous version of this paper there was an inaccuracy in the rounding off with $i$-th powers at the end of the proof of Proposition \ref{prop:main}. Many thanks to Matthew Tointon for encouraging us to state our main result in the generality of Theorem \ref{thm:main1}, for providing us with references to his work and for helpful feedback that led us to improve the presentation. We are grateful to Uri Shapira, Andreas Str{\"o}mbergsson and Mima Stanojkovski for feedback on a previous version of this work that led to an improvement of the presentation. The authors wish to thank the Max Planck Institute for Mathematics in Bonn for its financial support, great working conditions and an inspiring atmosphere. Daniel El-Baz is supported by the Austrian Science Fund (FWF), project Y-901. 

\section{Diameters of finite nilpotent groups} \label{sec:finnil}

\subsection{Diameters of a group and its quotients}
For a finite group $G$ with symmetric generating set $S$, a normal subgroup $H$ of $G$ and a normal subgroup $N$ of $H$, we view $H$ and $N$ as metric subspaces of the Cayley graphs $\Gamma(G, S)$, which allows us to define the diameters of $N$ and $H$ with respect to $S$, which we denote respectively by $\diam(H, S)$ and $\diam(N, S)$.

This metric also induces one on the quotient $\frac{H}{N}$ and allows us to define the diameter of that group with respect to the projections of the elements of $S$ onto $\frac{H}{N}$, which we denote by $\diam \left(\frac{H}{N}, S \right)$.

When $H=G$, that last quotient coincides as a metric space with the Cayley graph of $\frac{G}{N}$ with respect to the projections of the elements of $S$ in $\frac{G}{N}$.

The following lemma relates those three quantities.

\begin{lemma} \label{lem:quotients} For every finite group $G$ with symmetric generating set $S$, every normal subgroup $H$ of $G$ and every normal subgroup $N$ of $H$, we have 
\begin{equation} 
\diam \left(\frac{H}{N}, S \right) \le \diam(H, S) \le \diam \left(\frac{H}{N},S \right) + \diam(N,S).
\end{equation}
\end{lemma}

\begin{proof}
The lower bound on $\diam(H, S)$ follows from the definition.

We now prove the upper bound.
Fix two elements $h_1$ and $h_2$ in $H$.
Define $d_1 = \diam \left(\frac{H}{N}, S \right)$ and $d_2 = \diam(N, S)$.
By definition of $d_1$, we find $x_1, \ldots, x_{s_1} \in S$ with $s_1 \le d_1$ such that there exists $n \in N$ such that \begin{equation}\label{diamGmodN}
x_1 \cdots x_{s_1} h_1 = n h_2.
\end{equation}
For that $n \in N$ and by definition of $d_2$, we find $y_1, \ldots, y_{s_2} \in S$ with $s_2 \le d_2$ connecting $\mathrm{id}$ to $n$, that is 
\begin{equation}\label{diamN}
y_1 \cdots y_{s_2} = n.    
\end{equation}
Combining \eqref{diamGmodN} and \eqref{diamN}, we get
\begin{equation}
    y_{s_2}^{-1} \cdots y_1^{-1} x_1 \cdots x_{s_1} h_1 = h_2,
\end{equation}
which means that the distance between $h_1$ and $h_2$ via elements of $S$ is at most $s_1 + s_2$ which is itself at most $d_1 + d_2$, hence the claim.
\end{proof}

\subsection{Multilinear maps attached to groups} \label{tensors}
In this section we briefly recall (part of) the multilinear structure present on a group $G$.

For $x,y$ in $G$ we denote $[x,y]=xyx^{-1}y^{-1}$. Observe that $[x,y]^{-1}=[y,x]$. Furthermore if $z$ is also in $G$, then we have $[x,zy]=[x,z][x,y][z,[y,x]]^{-1}$. Observe that if $z,y$ are taken to be in $G^{(i)}$ for some $i \in \mathbb{Z}_{\geq 1}$, this last identity tells us that the commutator pairing
$$G \times \frac{G^{(i)}}{G^{(i+1)}} \to \frac{G^{(i+1)}}{G^{(i+2)}}
$$
is bilinear in the second entry (observe that by definition $\frac{G^{(j)}}{G^{(j+1)}}$ is an abelian group for any positive integer $j$). A similar computation shows that this pairing factors through $G^{(2)}$ in the first coordinate and is bilinear in both entries for the map
$$\frac{G^{(1)}}{G^{(2)}}  \times \frac{G^{(i)}}{G^{(i+1)}} \to \frac{G^{(i+1)}}{G^{(i+2)}}.
$$

Hence in total we get a homomorphism $(G^{\mathrm{ab}})^{\otimes i} \twoheadrightarrow \frac{G^{(i)}}{G^{(i+1)}}$, defined by the multilinear map from $(G^{\ab})^{i} \to \frac {G^{(i)}}{G^{(i+1)}}$ sending the vector $(g_1, \ldots, g_i)$ to the class of $[g_1, [g_2, \ldots, [g_{i-1}, g_i], \ldots]$ modulo $G^{(i+1)}$. In particular notice that if $S$ generates $G$, then nested commutators among elements of $S$ of length $i$ generate $\frac{G^{(i)}}{G^{(i+1)}}$.

\subsection{Comparing diameters}
We shall need the following elementary lemma. 
\begin{lemma} \label{elementary lemma}
Let $i$ be a positive integer. Then there exist positive integers $C_i,n_i$ such that for any $\lambda \in \mathbb{Z}_{\geq 1}$ one can find $a_1, \dots, a_{n_i},r$ in $\mathbb{Z}_{\geq 0}$ such that 
\[ 
\lambda=a_1^i+ \dots +a_{n_i}^i+r,
\]
with $r \leq C_i \lambda^{1/i}$. 
\begin{proof}
Observe that one can find a constant $D_i$ such that for each positive integer $\lambda$ one has a representation $\lambda=a_1^i+r_1$, with $r_1 \leq D_i\lambda^{\frac{i-1}{i}}$: to this end take $a_1:=\lfloor \lambda^{1/i} \rfloor$ and apply the binomial expansion to the worst case scenario $\lambda=(\lfloor \lambda^{1/i}\rfloor+1)^i-1$. Hence, iterating this, we obtain that for each $j$ in $\mathbb{Z}_{\geq 1}$ there are non-negative integers $a_1, \dots, a_j, r_j$ such that
\[ 
\lambda=a_1^i+ \dots +a_j^{i}+r_j,
\]
with $r_j \leq D_i^{\sum_{h=0}^{j-1}(\frac{i-1}{i})^h}\lambda^{(\frac{i-1}{i})^j}$. Choosing $j$ such that $\left(\frac{i-1}{i} \right)^j<\frac{1}{i}$ yields the desired conclusion. 
\end{proof}
\end{lemma}
\begin{prop} \label{prop:main}
For every finite group $G$, every symmetric generating set $S \subset G$ of size $s \in \Z_{\ge 2}$ and every $i \ge 1$, we have
\begin{equation}
    \diam\left(\frac{G^{(i)}}{G^{(i+1)}}, S \right) = O_{i, s} \left(\diam(\Gamma(G^{\ab}, S))^{1/i} \right).
\end{equation}
\end{prop}

\begin{proof}
We show that for each $h \in \frac{G^{(i)}}{G^{(i+1)}}$, there exist $y_1, \ldots, y_d$ in $S$ such that 
\begin{equation}
    h \equiv y_1 \cdots y_d \mod G^{(i+1)},
\end{equation}
with $d = O_{i,s} \left(\diam(\Gamma^{\ab}, S)^{1/i}) \right)$.
By \autoref{tensors}, we can find $(\lambda_f)_{f \colon [i] \to S} \in \Z^{[i] \to S}$ such that 
\begin{equation}
    h = \sum_{f \colon [i] \to S} \lambda_f [f(1), [ \ldots, [ , f(i)] \ldots ].
\end{equation}
By multilinearity, we can collect the last entries of the nested commutators for each choice of the first $i-1$ entries and obtain for each $g:[i-1] \to S$ an element $y_g$ in $G^{\ab}$ such that
\begin{equation} \label{eq:gsum}
    h = \sum_{g \colon [i-1] \to S}  [g(1), [ \ldots, [g(i-1) , y_g] \ldots ].
\end{equation}
Now for each $g:[i-1] \to S$ rewrite 
\begin{equation}
y_g=\sum_{x \in S} \lambda(x,g) [x]_{G^{\mathrm{ab}}}
\end{equation}
with 
\begin{equation}\label{eq:l1}
\sum_{x \in S} | \lambda(x, g) | \le \diam(\Gamma(G^{\ab}, S)).
\end{equation}

Using multilinearity once again, we have
\begin{equation}[g(1), [ \ldots, [g(i-1) , y_g] \ldots ] = \sum_{x \in S}  [g(1), [ \ldots, [g(i-1) ,  \lambda(x,g) [x]_{G^{\mathrm{ab}}}] \ldots ] 
\end{equation}
Rewrite, for each $x \in S$, using Lemma \ref{elementary lemma}
\begin{align*}
&\sgn(\lambda(x,g)) |\lambda(x,g)|  [g(1), [ \ldots, [g(i-1) ,   [x]_{G^{\mathrm{ab}}}] \ldots ]  \\
&= \sgn(\lambda(x,g))(\sum_{h=1}^{n_i}a_h^i+r) \cdot [g(1), [ \ldots, [g(i-1) ,   [x]_{G^{\mathrm{ab}}}] \ldots ] \\ &=\sgn(\lambda(x,g))((\sum_{h=1}^{n_i}[a_hg(1), [ \ldots, [a_hg(i-1) ,a_h   [x]_{G^{\mathrm{ab}}}] \ldots ])+r \cdot [g(1), [ \ldots, [g(i-1) ,   [x]_{G^{\mathrm{ab}}}] \ldots ]).
\end{align*}

On the last line, each of the $n_i+1$ summands are $O_i(| \lambda(x, g) |^{1/i})$, hence the $g$-th term of the sum \eqref{eq:gsum} has length at most $O_i(\sum_{x \in S} | \lambda(x, g) |^{1/i}))$, which is $O_i(\diam(\Gamma(G^{\ab}, S))^{1/i})$ by Jensen's inequality and recalling \eqref{eq:l1}.

Summing over all $g$ we thus get an upper bound of $s^{i-1} O_i(\diam(\Gamma(G^{\ab}, S))^{1/i})$, which is $O_{s,i}(\diam(\Gamma(G^{\ab}, S)^{1/i})$ as claimed.
\end{proof}

\begin{cor} \label{cor:main}
Let $G$ be a finite nilpotent group of class $c \in \Z_{\ge 1} $. Let $S \subset G$ be a symmetric generating set of size $s \in \mathbb{Z}_{\ge 2}$. We have
\begin{equation}
    \diam(\Gamma(G^\ab, S)) \le \diam(\Gamma(G,S)) \le \diam(\Gamma(G^\ab, S)) + O_{c,s}\left(\sqrt{\diam(\Gamma(G^\ab, S))} \right)
\end{equation}
\end{cor}

\begin{proof} 
By the left-hand side of the inequality in \autoref{lem:quotients}, the left-hand side follows immediately.

Using the right-hand side of the inequality in \autoref{lem:quotients} inductively for the terms of the lower central series, we obtain
\begin{equation}
    \diam(\Gamma(G,S)) \le \sum_{i \ge 1} \diam\left(\frac{G^{(i)}}{G^{(i+1)}}, S \right).
\end{equation}
Appealing to \autoref{prop:main} now yields the desired conclusion.
\end{proof}

\section{The case of unitriangular groups and more general sequences of nilpotent groups}
In this section, we apply \autoref{cor:main} to determine the limiting distribution of the appropriately rescaled diameters of random Cayley graphs of finite nilpotent groups of bounded rank and class.

The resulting theorem below is a generalisation of \cite[Theorem 1.2]{ShapiraZuck2018}, which corresponds to the case $c=1$ of our result. As the reader shall soon see, however, our proof consists of a reduction to that case by means of Corollary \ref{cor:main}.

\begin{thm} \label{thm:main1}
Let $k>r \geq 1$ and $c \geq 1$ be integers. Let $\{G_n\}_{n \in \mathbb{Z}_{\geq 1}}$ be a sequence of finite nilpotent groups of rank at most $r$, nilpotency class at most $c$ and with $|G_n|$ approaching infinity as $n$ goes to infinity.

Choosing a subset $S$ uniformly at random among all symmetric generating subsets $S$ of $G_n$ of size $k$, then as $n \to \infty$ we have that the random variables $\dfrac{\diam(\Gamma(G_n, S))}{|G_n^{\ab}|^{\frac{1}{k}}}$ converge in distribution.
Moreover,
\begin{equation}
   \frac{\diam(\Gamma(G_n, S))}{|G_n^{\ab}|^{\frac{1}{k}}} \xrightarrow[n \to \infty]{\mathrm{d}} 
   \diam(\R^k/L)
\end{equation}
where the random variable on the right-hand side is defined by choosing $L$ at random in the space $\SL_k(\R)/\SL_k(\Z)$ of unimodular lattices in $\R^k$ with respect to the Haar probability measure
and the diameter on the right-hand side is with respect to the $\ell^1$ metric.
\end{thm}

\begin{proof}
For $n \ge 1, r \ge 1$ and $k > r$, denote the random variable $\dfrac {\diam(\Gamma(G_n, S))}{|G_n^{\ab}|^{\frac{1}{k}}}$ by $X_n$ and the random variable $\dfrac {\diam(\Gamma(G_n^\ab, S))}{|G_n^{\ab}|^{\frac{1}{k}}}$ 
by $X_n^\ab$.

Applying \autoref{cor:main} to the finite nilpotent group $G_n$ (of class at most $c$) thus yields the inequalities 
\begin{equation} \label{squeezing}
    X_n^\ab \le X_n \le X_n^\ab +
    O_{k,c}\left(\frac{\sqrt{X_n^\ab}}{|G_n^{\ab}|^{\frac{1}{2k}}} \right).
\end{equation}
We next remark that we must have that $|G_n^{\ab}|$ approaches infinity as $n$ goes to infinity. Indeed from Section \ref{tensors} it follows immediately that there are finitely many nilpotent groups of bounded nilpotency class and bounded size of the abelianisation and this would contradict the fact that $|G_n|$ tends to infinity as $n$ goes to infinity. This is also explained in \cite[Lemma 4.13]{BreuillardTointon2016}. Moreover, observe that 
the groups $G_n^{\ab}$ trivially have rank bounded by $r$. We are therefore in a position to use \cite[Theorem 1.2]{ShapiraZuck2018} to deduce that $X_n^\ab$ converges in distribution to the random variable defined on the space of $k$-dimensional unimodular lattices as in the statement of our theorem, say $X$.

Note also that $\left(\frac{\sqrt{X_n^\ab}}{|G_n^{\ab}|^{\frac{1}{2k}}} \right)_n$ converges in probability to $0$.

The right-hand side of \eqref{squeezing} is therefore of the form $X_n^\ab + \varepsilon_n$ with $X_n^\ab \xrightarrow{\mathrm{d}} X$ and $\varepsilon_n \xrightarrow{\mathrm{P}} 0$.
It follows from Slutsky's lemma that $X_n^\ab + \varepsilon_n \xrightarrow{\mathrm{d}} X$.

This finishes the proof.
\end{proof}

We now use \autoref{thm:main1} with the the sequence of finite nilpotent groups $H_{q,d}$ of upper triangular $d \times d$ matrices over $\Z/q\Z$ with $1$ on the diagonal: they are of nilpotency class $d-1$ and $H_{q,d}^\ab \simeq (\Z/q\Z)^{d-1}$.
We thus obtain the following theorem and, in doing so, a proof of \cite[Conjecture 7]{HermonThomas2019} along with an explicit description of the limiting distribution.

\begin{thm} \label{thm:main}
Let $q \ge 2, d \ge 2$ and $k \ge d$. 
Choosing a subset $S$ uniformly at random among all symmetric generating subsets $S$ of $H_{q,d}$ of size $k$, then as $q \to \infty$ we have that the random variables $\dfrac{\diam(\Gamma(H_{q,d}, S))}{q^{(d-1)/k}}$ converge in distribution.
Moreover,
\begin{equation}
   \frac{\diam(\Gamma(H_{q,d}, S))}{q^{(d-1)/k}} \xrightarrow[q \to \infty]{\mathrm{d}} 
   \diam(\R^k/L)
\end{equation}
where the random variable on the right-hand side is defined by choosing $L$ at random in the space $\SL_k(\R)/\SL_k(\Z)$ of unimodular lattices in $\R^k$ with respect to the Haar probability measure
and the diameter on the right-hand side is with respect to the $\ell^1$ metric.
\end{thm}

\section{Concluding remarks}

Let $i$ be in $\mathbb{Z}_{\geq 2}$. One can then ask the following related questions. 

\textbf{Questions:} What is the correct order of magnitude of $\diam \left(H_{q,d}^{(i)},S \right)$? Is the power $q^{\frac{d-1}{ik}}$ suggested by the upper bound of \autoref{prop:main} sharp (to hold in probability)? 

We only remark that, using the same type of argument based on growth that one uses to show the logarithmic behaviour as a general lower bound, one can establish as a pointwise lower bound a much smaller power of $q$, depending on $i$. Such a trivial estimate can be slightly improved using the equidistribution theorem in \cite{ShapiraZuck2018} and basic facts about the shortest vector statistics on spaces of unimodular lattices.  However, the resulting gain on the power of $q$ is still not enough to reach $q^{\frac{d-1}{ik}-\epsilon}$ as a pointwise lower bound. 

One can also ask about the difference $\diam \left(\Gamma(H_{q,d},S) \right)-\diam \left(\Gamma(H_{q,d}^{\ab},S) \right)$. Corollary \ref{cor:main} gives an upper bound for this quantity. We ask the following: 

\textbf{Question:} Can one give a sharp lower bound for the quantity $\diam \left(\Gamma(H_{q,d},S) \right)-\diam \left(\Gamma(H_{q,d}^{\ab},S) \right)$ (to hold in probability)? 

Finally, what about those questions for more general sequences $\{G_n\}_{n \in \mathbb{Z}_{\geq 1}}$ of finite nilpotent groups as in Theorem \ref{thm:main1}?
\bibliographystyle{plain}
\bibliography{references}

\end{document}